\documentclass[12pt,a4paper]{amsart}
\usepackage{amsmath,amsthm,amssymb}

\usepackage[shortlabels]{enumitem}

\usepackage[totalwidth=15.75cm,totalheight=22.275cm]{geometry}

\usepackage{graphicx}

\usepackage{hyperref}
\usepackage[usenames, dvipsnames]{color}
\usepackage[flushmargin]{footmisc}

\usepackage{gensymb}

\numberwithin{equation}{section}

\makeatletter
\def\swappedhead#1#2#3{%
  \thmnumber{\@upn{\the\thm@headfont #2\@ifnotempty{#1}{.~}}}%
  \thmname{#1}%
  \thmnote{ {\the\thm@notefont(#3)}}}
\makeatother
\swapnumbers

\theoremstyle{plain}
\newtheorem{thm}{Theorem}[section]%
\newtheorem{prop}[thm]{Proposition}%
\newtheorem{lem}[thm]{Lemma}%
\newtheorem{cor}[thm]{Corollary}%
\theoremstyle{definition}
\newtheorem{question}[thm]{Question}%
\newtheoremstyle{claimstyle}%
   {}
   {}
   {\normalfont}
   {}
   {\itshape}
   {.}
   { }
   {\thmnote{#3}}
   
\theoremstyle{claimstyle}
\newtheorem*{varclaim}{}

\newenvironment{remark}[1][Remark]{\begin{varclaim}[#1]}{\end{varclaim}}

\newenvironment{sketch}{\begin{proof}[Sketch of proof]}{\end{proof}}

\newcommand*{\defeq}{\mathrel{\vcenter{\baselineskip0.5ex \lineskiplimit0pt
                     \hbox{\scriptsize.}\hbox{\scriptsize.}}}%
                     =}

\newcommand{\eps}{\varepsilon}
\renewcommand{\theta}{\vartheta}
\renewcommand{\phi}{\varphi}

\newcommand{\C}{{\mathbb{C}}}
\newcommand{\Ch}{\hat{\C}}

\newcommand{\R}{{\mathbb{R}}}

\newcommand{\dist}{\operatorname{dist}}

\newcommand{\UO}{\operatorname{UO}}
\newcommand{\BU}{\operatorname{BU}}

\title{Escaping sets are not sigma-compact}

\begin{document}

\author{Lasse Rempe} 
\address{Dept. of Mathematical Sciences \\
	 University of Liverpool \\
   Liverpool L69 7ZL\\
   UK \\ 
	 \textsc{\newline \indent 
	   \href{https://orcid.org/0000-0001-8032-8580%
	     }{\includegraphics[width=1em,height=1em]{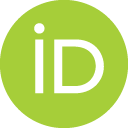} {\normalfont https://orcid.org/0000-0001-8032-8580}}
	       }}
\date{\today}
\email{l.rempe@liverpool.ac.uk}
\subjclass[2010]{Primary 30D05, Secondary 37F10, 54D45.}

\begin{abstract}Let $f$ be a transcendental entire function. The \emph{escaping set} $I(f)$ consists of those points that tend to infinity under iteration of $f$. We show that
  $I(f)$ is not $\sigma$-compact, resolving a question of Rippon from 2009.
  \end{abstract}

\maketitle

\section{Introduction}
 Let $f\colon \C\to\C $ be a transcendental entire function. The set 
    \[ I(f)\defeq \{z\in \C\colon \lim_{n\to\infty} f^n(z) = \infty\} \]
  is called the \emph{escaping set} of $f$, where 
     \[ f^n = \underbrace{f\circ\dots\circ f}_{n\text{ times}}\] 
    denotes the $n$-th iterate of $f$. The escaping set  was first studied by Eremenko~\cite{alexescaping} and has been the subject of intensive research in recent years;
     see e.g.~\cite{oberwolfach,rrrs,bergweilermeasure,ripponstallarderemenkopoints} and their references. The topological study of $I(f)$ turns out to be surprisingly
     intricate. For example Eremenko \cite[p.~343]{alexescaping} asked whether, for every 
     transcendental entire function, every
     connected component of $I(f)$ is unbounded. This question has become known as \emph{Eremenko's conjecture}, and is one of the most famous open problems
     in transcendental dynamics. Strengthened versions of the conjecture have since been proved false in general; compare~\cite[Theorem~1.1]{rrrs} 
     and~\cite[Theorem~1.6]{arclike}. 
     On the other hand, Rippon and Stallard~\cite{boundariesescaping} 
     have shown that $I(f)\cup\{\infty\}$ is always connected, so any counterexample to Eremenko's original conjecture would have rather
     subtle topological properties.
     
   The \emph{fast escaping set} is a certain subset of the escaping set, first introduced by Bergweiler and Hinkkanen~\cite{bergweilerhinkkanen}. This set has also received considerable attention recently; in part because it 
    appears to be more tractable than the full escaping set 
      $I(f)$. In particular, the analogue of Eremenko's conjecture holds for $A(f)$ \cite{ripponstallardfatoueremenko}. One difference between the two
      objects lies in the topological complexity of their definitions. Indeed, $A(f)$ can be written as a countable union of closed sets~\cite[Formula~(1.3)]{ripponstallardfast}; 
       so it is an $F_{\sigma}$ set. On the other hand, by definition 
      \begin{align*} I(f) &= \{z\in\C\colon \forall M\geq 0\, \exists N\geq 0 \colon \lvert f^n(z)\rvert \geq M \text{ for $n\geq N$}\} \\ 
            &=\bigcap_{M= 0}^{\infty} \bigcup_{N=0}^{\infty}  \bigcap_{n=N}^{\infty} f^{-n}(\C\setminus D(0,M)), \end{align*}
    where $D(0,M)$ denotes the disc of radius $M$ around $0$.  
      So $I(f)$ is $F_{\sigma \delta}$: a countable intersection of countable unions of closed sets.       
     This raises the following question, posed by Rippon in 2009~\cite[Problem~8, p.~2960]{oberwolfach}.
     
     \begin{question}\label{qu:Fsigma}
       Is there a transcendental entire function $f$ for which $I(f)$ itself is $F_{\sigma}$? 
     \end{question}
         Since any closed subset of $\C$ is a countable union of compact sets, it is equivalent to ask whether 
         $I(f)$ is ever \emph{$\sigma$-compact} (a countable union of compact sets); compare also \cite{lipham}. It is well-known that 
        $I(f)$ cannot be a $G_{\delta}$ (a countable intersection of
        open sets); see Lemma~\ref{lemma:Gdelta} below. 
      
   As far as we are aware, there is no function $f$ for which Question~\ref{qu:Fsigma} has been previously resolved; the goal of this paper is to
     give a negative answer in general. In fact, we prove a stronger result, as follows.
     Let $\UO(f)$ consist of all $z\in\C$ whose orbit $\{f^n(z)\colon n\geq 0\}$ is unbounded, 
     and let $\BU(f) = \UO(f)\setminus I(f)$ denote the ``bungee set''; see \cite{bungee}. Recall that the \emph{Julia set} $J(f)$ is the set of non-normality
     of the iterates of~$f$. 
     \begin{thm}\label{thm:main}
       Let $f$ be a transcendental entire function. 
        Then every $\sigma$-compact subset of $\UO(f)$ omits some point of $I(f)\cap J(f)$ and some point of $\BU(f)\cap J(f)$. 
                 
      In particular, the sets $I(f)$, $\UO(f)$, $\BU(f)$ and their intersections with $J(f)$ are not $\sigma$-compact. 
     \end{thm}
          
     \subsection*{Acknowledgements} I thank David Lipham and Phil Rippon for interesting discussions. I also thank David Mart\'i-Pete, James Waterman and the referee for 
          helpful comments that improved
  		the presentation of the paper.

\section{Proof of the Theorem}
  We require a result on the existence of arbitrarily slowly escaping points, 
   which follows from work of Rippon and Stallard~\cite{ripponstallardannular}; see also~\cite[Theorem~1]{ripponstallardslow}. 
   \begin{thm}\label{thm:slow}
      Let $f$ be a transcendental entire function, and let $R_0\geq 0$. Then there exists $M_0>R_0$ 
       with the following property. 
        If $(a_m)_{m=0}^{\infty}$ is a sequence with $a_m\to\infty$ and $a_m\geq M_0$ for all $m$, 
        then there are $\zeta \in J(f)\cap I(f)$ and $\omega\in J(f)\cap \BU(f)$ with
              \[ R_0 \leq \lvert f^m(\zeta)\rvert \leq a_m\quad\text{and}\quad R_0 \leq \lvert f^m(\omega)\rvert \leq a_m \quad\text{for $m\geq 0$}. \]
   \end{thm}
    \begin{proof}
    The result is an easy consequence of~\cite{ripponstallardannular},
      which studies ``annular itineraries'' of entire functions. 
      Fix some $R>0$ such that $M(r)>r$ for $r\geq R$, where $M(r)$
      denotes the maximum modulus of $f$ on the disc $\overline{D(0,r)}$. 
      Then the \emph{annular itinerary} of a point $z$ is the sequence 
        $(s_m)_{m=0}^{\infty}$ such that
        $s_m=0$ if $\lvert f^m(z)\rvert < R$, and 
        \[ M^{s_m-1}(R) \leq \lvert f^m(z) \rvert < M^{s_m}(R) \] 
        otherwise. (Here $M^k$ is the $k$-th iterate of $r\mapsto M(r)$.)
        
      Theorem 1.2 of~\cite{ripponstallardannular} implies that for suitable $R\geq R_0$, for an entry $s_m=n\geq 0$ in an annular itinerary, 
       there is a subset $X_n\subset \{0,\dots,n+1\}$ of allowable next entries $s_{m+1}$, and that, for infinitely many $n$, all entries but at most one
       are allowable. More precisely,  
           \begin{enumerate}[(a)]
              \item any sequence $(s_m)_{m=0}^{\infty}$ with $s_{m+1}\in X_{s_m}$ for all $m$ is the annular itinerary of some $z\in J(f)$;\label{item:admissible}
              \item $n+1\in X_n$ for all $n$; 
              \item there is an increasing sequence $(n_j)_{j=1}^{\infty}$ such that $\# X_{n_j} \geq n+1$.
            \end{enumerate}
       (With the notation of~\cite{ripponstallardannular}, $X_{n_j}=\{0,\dots,n_j+1\}\setminus I_j$, where $\# I_j \leq 1$, and
           $X_n = \{n+1\}$ when $n\neq n_j$ for all $j$.) We may suppose that the sequence $(n_j)$ is chosen such that $n_1 \geq 2$.

     Set $M_0\defeq M^{n_1}(R)$ and 
      let $(a_m)_{m=0}^{\infty}$ be a sequence as in the statement of the theorem. We may assume that $(a_m)$ is non-decreasing. 
      Similarly as in the proof of \cite[Corollary~1.3~(d)]{ripponstallardannular}, we can construct 
         an annular itinerary $(s_m)_{m=0}^{\infty}$ satisfying~\ref{item:admissible} such that $s_m\to\infty$ and 
         \begin{equation}\label{eqn:slowgrowth} M^{s_m}(R) \leq a_m \qquad\text{for } m\geq 0. \end{equation}
       Indeed, for each $j$, either $n_j\in X_{n_j}$ or $n_{j}-1\in X_{n_j}$; let us denote this element of $X_{n_j}$ by $\tilde{n}_j$. Then any sequence of the form
        \[ (s_m)_{m=0}^{\infty} = \underbrace{n_1, \tilde{n}_1, n_1, \tilde{n}_1,\dots, n_1}_{\text{length }N_1},n_{1}+1, n_1+2, \dots, 
            \underbrace{n_{2}, \tilde{n}_{2}, \dots, n_{2}}_{\text{length }N_2} ,n_{2} + 1, n_{2} + 2, \dots \]
      satisfies~\ref{item:admissible}. If $N_j$ is chosen sufficiently large that 
               \begin{equation}\label{eqn:Nj} a_m \geq M^{n_{j+1}}(R) \quad\text{for } m\geq N_j, \end{equation}
        then~\eqref{eqn:slowgrowth} holds for $m\geq N_1$. It also holds for $m\leq N_1$, since $s_m\leq n_1$ for these values of $m$, and 
     $a_m \geq M_0 = M^{n_1}(R)$ by assumption. 
      
      Let $\zeta \in J(f)$ have annular itinerary $(s_m)_{m=0}^{\infty}$; then $\zeta\in I(f)$ and 
         \[
            R_0\leq R \leq M^{s_m-1}(R) \leq \lvert f^m(\zeta)\rvert  < M^{s_m}(R) \leq a_m \quad \text{for all $m$}. \]
       Hence 
       $\zeta$ has the desired properties.   
   To obtain $\omega$, we instead use an unbounded but non-escaping sequence of the form 
        \[(\tilde{s}_m)_{m=0}^{\infty} = B_1,n_{1}+1, n_1+2, \dots, n_{2}, 
            \hat{n}_2, B_2 , n_1+1,n_1+2,\dots,n_3,\hat{n}_3, B_3,\dots, \]
     where $\hat{n}_j\in X_{n_j}$ is either $n_1$ or $n_1-1$, and $B_j$ is a block alternating between entries 
      $n_1$ and $\tilde{n}_1$. To ensure that the sequence satisfies~\ref{item:admissible}, $B_j$ should end in $n_1$; it should begin
      with $n_1$ if $\hat{n}_j=n_1-1$, and with $\tilde{n}_1$ if $\hat{n}_j = n_1$. The length $N_j$ of the block $B_j$ is again chosen to satisfy~\eqref{eqn:Nj}.
    \end{proof}

\begin{proof}[{Proof of Theorem~\ref{thm:main}}]
 If $D\subset\C$ is closed and $z\in\C$, we set 
    \[ n_D(z) \defeq \min\{n\geq 0\colon f^n(z)\notin D\} \leq \infty, \]
     with the convention that $n_D(z)=\infty$ if no such $n$ exists. Observe that
       $n_D$ depends upper semicontinuously on $z$, as the infimum of
       upper semicontinuous functions. Indeed, 
        \[ n_D(z) = \min\{\chi_n(z) \colon n \geq 0\}, \qquad\text{where}\qquad
            \chi_n(z) = \begin{cases}
                             \infty & \text{if } z\in f^{-n}(D) \\
                             n & \text{if } z\notin f^{-n}(D),\end{cases} \]
      and each $\chi_n$ is upper semicontinuous since $f^{-n}(D)$ is closed. 
      Note also that 
     $z\in \UO(f)$ if and only if $n_D(z)<\infty$ for every compact $D\subset\C$. 

  Now let $R_0 \geq 0$ be arbitrary, and choose $M_0$ as in Theorem~\ref{thm:slow}.
   Let $X\subset \UO(f)$ be $\sigma$-compact; say 
   $X = \bigcup_{j=0}^{\infty} K_j$ where each $K_j$ is compact. 
     Define $M_j \defeq M_0 +j$ and $D_j \defeq \overline{D(0,M_j)}$. Then
     $n_{D_j}(z)<\infty$ for every $z\in K_j$. 
     
     Since $K_j$ is compact and $n_{D_j}$ is
     upper semicontinuous,  
     \[ n_j \defeq \max_{z\in K_j} n_{D_j}(z) \]
     exists for every $j$. Set $N_0\defeq n_0$ and $N_{j+1}\defeq \max\{n_{j+1},N_j+1\}$. 
     For $n\geq 0$, define $j(n)\defeq\min\{j\colon N_j\geq n\}$ and 
      \[ a_n \defeq M_{j(n)} = \min\{ M_j \colon  N_j\geq n \}. \]
     Then $a_n\geq M_0$ for all $n$ and $\lim_{n\to\infty} a_n = \infty$. So by Theorem~\ref{thm:slow}, there are $\zeta\in I(f)\cap J(f)$ and $\omega\in \BU(f)\cap J(f)$ such that
       $\lvert f^n(z)\rvert \leq a_n$ for $z\in \{\zeta,\omega\}$ and all $n\geq 0$.
       
      Let $j\geq 0$. Then, for $n\leq n_j\leq N_j$, we have
        $\lvert f^n(z)\rvert \leq a_n \leq M_j$, and hence $f^n(z) \in D_j$. So $n_{D_j}(z) > n_j$, and $z\notin K_j$ by choice of $n_j$. 
       Thus $z\notin X = \bigcup_{j=0}^{\infty} K_j$, as claimed. 
\end{proof}

By the expanding property of the Julia set, we also obtain the following answer to a question of Lipham (personal communication).

\begin{cor}\label{cor:nowheresigmacompact}
  Let $f$ be a transcendental entire function and let $Y$ be  one of the sets $I(f)\cap J(f)$, $\UO(f)\cap J(f)$ and $\BU(f)\cap J(f)$. 
   Then $Y$ is nowhere $\sigma$-compact. That is, if $X\subset Y$ is $\sigma$-compact, then
    $X$ does not contain a non-empty relatively open subset of $Y$. 
\end{cor} 
\begin{proof}
   We prove a  slightly stronger statement. Let $Y$ be as in the statement and  let  
   $X\subset \UO(f)\cap J(f)$ be $\sigma$-compact. If $U\subset\C$ is open with $U\cap Y \neq \emptyset$, then $X$ omits some point of 
      $U\cap Y$. 
      
    Let $R_0$ and $M_0$ be as in Theorem~\ref{thm:slow}, and set $L_0\defeq J(f)\cap \overline{D(0,M_0)}\setminus D(0,R_0)$. If $R_0$ is chosen sufficiently large, then 
      $L_0$ contains no Fatou exceptional point of $f$. (Recall that a Fatou exceptional point is a point with finite backward orbit, 
      and that $f$ has at most one finite exceptional point \cite[p.~156]{waltersurvey}.) Since $U\cap J(f)\neq \emptyset$, it follows that 
     there is some $n$ such that $f^n(U)\supset L_0$.
     
     The proof of Theorem~\ref{thm:main} shows that no $\sigma$-compact subset of $\UO(f)\cap J(f)$ contains $L_0\cap Y$. (Note that we may take $a_0=M_0$ in the proof
      to ensure that $\lvert \zeta \rvert,\lvert \omega\rvert \leq M_0$.) Moreover, $f^n(X)$ is $\sigma$-compact as the 
	image of a $\sigma$-compact set under a continuous function. So $f^n(X)$ omits some points of $L_0\cap Y$. Since $Y$ is backward-invariant,
	$X$ omits some points of $U\cap Y$, as claimed. 
\end{proof} 

We remark that the proof of Theorem~\ref{thm:main} establishes the following very general principle: 
\begin{prop}[Abstract version of the main theorem]\label{prop:abstract}
  Let $f\colon U\to V$ be continuous, where $U,V$are non-empty topological spaces and $U$ is
  $\sigma$-compact. Let $\UO(f)$ denote the set of $z\in U$ such that $f^n(z)$ is defined and in $U$ for all $n\geq 0$, but
   the orbit $\{f^n(z)\}$ is not contained in any compact subset of $U$.
   
  Suppose that $X\subset \UO(f)$ is $\sigma$-compact and $\Delta\subset U$ is compact. Then there is a sequence 
  $(\Delta_n)_{n=0}^{\infty}$ of compact subsets $\Delta_n\subset U$ with $\Delta\subset \Delta_0\subset \Delta_1\subset \dots$ and 
    $\bigcup_{n=0}^{\infty} \Delta_n = U$ such that $X$ contains no point $\zeta\in \UO(f)$ with
  $f^n(\zeta)\in \Delta_n$ for all $n\geq 0$. 
  \end{prop}
  \begin{remark}
    Note that we are not assuming any relation between the spaces $U$ and $V$. However, the statement is vacuous when $\UO(f)=\emptyset$, and in
      particular when $U\cap V\neq \emptyset$. 
  \end{remark}
  \begin{sketch}
     Let $(D_j)_{j=0}^{\infty}$ be an increasing sequence of compact subsets $D_j\subset U$ with $\bigcup _{j=0}^{\infty} D_j= U$. 
      Define $j(n)$ as in the proof of Theorem~\ref{thm:main}; then the sets $\Delta_n \defeq D_{j(n)}$ have the desired property. 
  \end{sketch}

  It follows that the set of escaping points is not $\sigma$-compact in any
  setting where an analogue of  Theorem~\ref{thm:slow} holds. 
  This includes:
  \begin{enumerate}[(a)]
    \item transcendental meromorphic functions $\C\to\Ch$ \cite{ripponstallardslow};\label{item:mero}
    \item transcendental self-maps of the punctured plane \cite[Theorem~1.2]{martipeteCstar};\label{item:Cstar}
    \item quasiregular self-maps $f\colon \R^d\to\R^d$ of transcendental type
        \cite{nicksslow}; 
    \item continuous functions $\phi\colon [0,\infty)\to [0,\infty)$ with $\phi(t)\not\to\infty$ as $t\to\infty$,
       and such that $I(\phi)\neq \emptyset$ \cite[Theorem~2.2]{osborneripponstallard}.
  \end{enumerate}

   A more general setting than both~\ref{item:mero} and \ref{item:Cstar} is provided by
     the \emph{Ahlfors islands maps} of Epstein; see e.g.~\cite{exotic}. 
     These are maps $f\colon W\to X$, where $X$ is a compact one-dimensional
     manifold, $W\subset X$ is open, and $f$ satisfies certain transcendence 
     conditions near $\partial W$. We may define the escaping set $I(f)$ 
     as the set of points $z\in W$ with $f^n(z)\in W$ for all $n$ and
     $\dist(f^n(z),\partial W)\to 0$ as $n\to\infty$. It is plausible that an analogue of Theorem~\ref{thm:slow} holds for
     Ahlfors islands maps with $W\neq X$, using a similar proof
    as in~\cite{ripponstallardslow}. This would mean, by Proposition~\ref{prop:abstract}, that $I(f)$ is not 
    $\sigma$-compact for such functions.

 For completeness, we conclude by giving the simple proof that $I(f)$ is never a $G_{\delta}$ set; compare also \cite[Corollary~3.2]{liphametds}. 
\begin{lem}\label{lemma:Gdelta}
 Let $f$ be a transcendental entire function. Then $I(f)$ and $I(f)\cap J(f)$ are not $G_{\delta}$ sets. 
\end{lem}
\begin{proof}
  Every closed subset of $\C$ (or any metric space) is $G_{\delta}$; so $J(f)$ is $G_{\delta}$. 
   The intersection of two $G_{\delta}$ sets is again $G_{\delta}$, so it is enough to prove the claim for $I(f)\cap J(f)$. 
 By \cite[Theorem~2]{alexescaping}, $I(f)\cap J(f)$ is nonempty, and hence dense in $J(f)$ by Montel's theorem. 
  By Baire's theorem, any two dense $G_{\delta}$ subsets of $J(f)$ must intersect. Hence it is enough
    to observe that $\BU(f)\cap J(f)$ contains a dense $G_{\delta}$ by Montel's theorem, namely the set of points whose orbits are dense in $J(f)$.  
      (See \cite[Lemma~1]{bakerdominguezresidual}.) 
\end{proof}

 Lipham has pointed out the following reformulation of Corollary~\ref{cor:nowheresigmacompact}.
 \begin{cor}
    Any $G_{\delta}$ set 
    $A\subset Y\defeq J(f)\setminus I(f)$ is nowhere dense in $Y$. 

    In particular, $Y$ is $G_{\delta \sigma}$ but not \emph{strongly $\sigma$-complete}; that is, it cannot be written as a countable
    union of relatively closed $G_{\delta}$ subsets.
  \end{cor}
  \begin{proof}
    Suppose, by contradiction, that $\overline{D(\zeta,\eps)}\cap Y \subset A$ for some $\zeta\in J(f)$ and $\eps>0$. The complement
      of a $G_{\delta}$ set is $F_{\sigma}$, so 
    $B \defeq J(f)\cap \overline{D(\zeta,\eps)} \setminus A$ is an $F_{\sigma}$ set 
     with \[ D(\zeta,\eps)\cap I(f)\cap J(f)\subset B \subset I(f)\cap J(f),\] which contradicts Corollary~\ref{cor:nowheresigmacompact}. 

   The set $Y$ is $G_{\delta \sigma}$ by definition. 
     As mentioned in the proof of Lemma~\ref{lemma:Gdelta}, $Y$ contains a $G_{\delta}$ subset $U$ that is dense in $J(f)$.
     On the other hand, if $(A_k)_{k=0}^{\infty}$ is a sequence of relatively closed $G_{\delta}$ subsets of $Y$, then 
     $\overline{A_k}$ is closed and nowhere dense in $J(f)$. Hence
        \[  \bigcup_{k=0}^{\infty} A_k \subset \bigcup_{k=0}^{\infty} \overline{A_k} \not\supset U \subset Y, \]
      by Baire's theorem, and $Y$ is indeed strongly $\sigma$-complete. 
  \end{proof}


\begin{thebibliography}{RRRS11}

\bibitem[BD00]{bakerdominguezresidual}
I.~N. Baker and P.~Dom\'{\i}nguez, \emph{Residual {J}ulia sets}, J. Anal.
  \textbf{8} (2000), 121--137.

\bibitem[Ber93]{waltersurvey}
Walter Bergweiler, \emph{Iteration of meromorphic functions}, Bull. Amer. Math.
  Soc. (N.S.) \textbf{29} (1993), no.~2, 151--188.

\bibitem[Ber18]{bergweilermeasure}
Walter Bergweiler, \emph{Lebesgue measure of {J}ulia sets and escaping sets of
  certain entire functions}, Fund. Math. \textbf{242} (2018), no.~3, 281--301.
  
\bibitem[BH99]{bergweilerhinkkanen}
Walter Bergweiler and Aimo Hinkkanen, \emph{On semiconjugation of entire functions}, Math. Proc. Cambridge Philos. Soc. \textbf{126} (1999), no. 3, 565--574. 

\bibitem[Ere89]{alexescaping}
A.~\`E. Eremenko, \emph{On the iteration of entire functions}, Dynamical
  systems and ergodic theory ({W}arsaw, 1986), Banach Center Publ., vol.~23,
  PWN, Warsaw, 1989, pp.~339--345.

\bibitem[Lip20a]{liphametds}
David~S. Lipham, \emph{A note on the topology of escaping endpoints}, Ergodic
  Theory and Dynamical Systems (2020, First View), 1--4.

\bibitem[Lip20b]{lipham}
David~S. Lipham, \emph{The topological dimension of radial Julia sets},
  Preprint arXiv:2002.00853, 2020.

\bibitem[MP18]{martipeteCstar}
David Mart\'{\i}-Pete, \emph{The escaping set of transcendental self-maps of
  the punctured plane}, Ergodic Theory Dynam. Systems \textbf{38} (2018),
  no.~2, 739--760.

\bibitem[Nic16]{nicksslow}
Daniel~A. Nicks, \emph{Slow escaping points of quasiregular mappings}, Math. Z.
  \textbf{284} (2016), no.~3-4, 1053--1071.

\bibitem[Obe09]{oberwolfach}
\emph{Mini-{W}orkshop: {T}he {E}scaping {S}et in {T}ranscendental {D}ynamics},
  Report on the mini-workshop held December 6--12, 2009,
  organized by Walter Bergweiler and Gwyneth Stallard. Oberwolfach Reports.
  Vol. 6 (2009), no. 4, pp.~2927--2963.

\bibitem[ORS19]{osborneripponstallard}
John~W. Osborne, Philip~J. Rippon, and Gwyneth~M. Stallard, \emph{The iterated
  minimum modulus and conjectures of {B}aker and {E}remenko}, J. Anal. Math.
  \textbf{139} (2019), no.~2, 521--558.

\bibitem[OS16]{bungee}
John~W. Osborne and David~J. Sixsmith, \emph{On the set where the iterates of
  an entire function are neither escaping nor bounded}, Ann. Acad. Sci. Fenn.
  Math. \textbf{41} (2016), no.~2, 561--578.
  
\bibitem[R16]{arclike}
Lasse Rempe, \emph{Arc-like continua, Julia sets of entire functions, and Eremenko's conjecture}, 
 Preprint arXiv:1610.06278, 2016.

\bibitem[RR12]{exotic}
Lasse Rempe and Philip~J. Rippon, \emph{Exotic {B}aker and wandering domains
  for {A}hlfors islands maps}, J. Anal. Math. \textbf{117} (2012), 297--319.

\bibitem[RRRS11]{rrrs}
G\"unter Rottenfu{\ss}er, Johannes R\"uckert, Lasse Rempe, and Dierk
  Schleicher, \emph{Dynamic rays of bounded-type entire functions}, Ann. of
  Math. (2) \textbf{173} (2011), no.~1, 77--125.
  
\bibitem[RS05]{ripponstallardfatoueremenko}
P.~J. Rippon and G.~M. Stallard, \emph{On questions of Fatou and Eremenko}, Proc. Amer. Math. Soc. \textbf{133} (2005), no. 4, 1119–1126.

\bibitem[RS11a]{ripponstallardslow}
\bysame, \emph{Slow escaping points of meromorphic
  functions}, Trans. Amer. Math. Soc. \textbf{363} (2011), no.~8, 4171--4201.
  
\bibitem[RS11b]{boundariesescaping}
\bysame, \emph{Boundaries of escaping Fatou components}, Proc. Amer. Math.
  Soc. \textbf{139} (2011), no.~8, 2807--2820.

\bibitem[RS12]{ripponstallardfast}
\bysame, \emph{Fast escaping points of entire functions}, Proc. Lond. Math. Soc. \textbf{105} (2012), no. 4, 787--820.

\bibitem[RS15]{ripponstallardannular}
\bysame, \emph{Annular itineraries for entire functions}, Trans. Amer. Math.
  Soc. \textbf{367} (2015), no.~1, 377--399.
  


\bibitem[RS19]{ripponstallarderemenkopoints}
\bysame, \emph{Eremenko points and the structure of the escaping set}, Trans.
  Amer. Math. Soc. \textbf{372} (2019), no.~5, 3083--3111.

\end{thebibliography}
  
\providecommand{\href}[2]{#2}

\end{document}